\documentclass[amssymb,12pt]{article}

\usepackage{mathtools}
\usepackage{geometry,amssymb,amsthm,amsmath,
graphics,enumerate}
\usepackage{times}
\usepackage{amsfonts}
\usepackage{amscd}
\usepackage{url}
\usepackage{enumitem}
\author{ Michael C.\
Laskowski 
\\
Department of Mathematics\\University of Maryland
}
\newbox\smilebox
\newbox\anchorbox
\newbox\noanchorbox
\newbox\tempbox

\setbox\smilebox=\hbox{$\smile$}

\def\anchor{\hbox{\vtop{
           \hbox to \wd\smilebox{\hfil\vrule width.4pt height7pt depth1pt\hfil}
           \vskip  -11.5truept
           \hbox to \wd\smilebox{\hfil$\smile$\hfil}}}}
\setbox\anchorbox=\anchor
\def\noanchor{\hbox{\vtop{
           \hbox to \wd\anchorbox{\hfil\anchor\hfil}
           \vskip -14truept
           \hbox to \wd\anchorbox{\hfil/\hfil}}}}
\setbox\noanchorbox=\noanchor

\def\fg#1#2#3{\setbox\tempbox=\hbox{$\scriptstyle{#2}$}
\ifnum\wd\anchorbox>\wd\tempbox\dimen255=\wd\anchorbox
\else\dimen255=\wd\tempbox\fi
{#1\,\vtop{\hbox to \dimen255{\hfil\anchor\hfil}
           \vskip -6truept
           \hbox to \dimen255{\hfil$\scriptstyle{#2}$\hfil}}
           \,#3}}

\def\nfg#1#2#3{\setbox\tempbox=\hbox{$\scriptstyle{#2}$}
\ifnum\wd\noanchorbox>\wd\tempbox\dimen255=\wd\noanchorbox
\else\dimen255=\wd\tempbox\fi
{#1\,\vtop{\hbox to \dimen255{\hfil\noanchor\hfil}
           \vskip -6truept
           \hbox to \dimen255{\hfil$\scriptstyle{#2}$\hfil}}
           \,#3}}

\setbox1=\hbox{$\bot$}

\def\north#1#2{#1\,
\hbox{$\bot$\llap {\hbox to\wd1 {\hfil $/$\hfil}}}
\,#2}

\def\nao#1#2#3{#1\  \hbox{\vtop{ 
\baselineskip=4pt
\hbox{$\bot$\llap {\hbox to\wd1 {\hfil $/$\hfil}}
\hskip .05em \llap{\hbox{$^{\scriptscriptstyle{a}}$}}}\hbox{$\scriptstyle
{#2}$}}}\, #3}

\def\includeE#1{{\lhook\kern-3.5pt\joinrel\smash{
    \mathop{\longrightarrow}\limits^{#1}}}}

\def\efor/{Example~\ref{E4}}

\def\BL/{Baldwin--Lachlan}
\def\Bu/{Buechler}
\def\Hr/{Hrushovski}
\def\lm/{locally modular}
\def\wm/{weakly minimal}
\def\nm/{non--modular}
\def\ss/{superstable}
\def\ud/{unidimensional}
\def\sm/{strongly minimal}

\def\abar{\overline{a}}

\def\bbar{\overline{b}}
\def\cbar{\overline{c}}
\def\dbar{\overline{d}}
\def\ebar{\overline{e}}
\def\fbar{\overline{f}}

\def\hbar{\overline{h}}

\def\kbar{\overline{k}}

\def\mbar{\overline{m}}

\def\xbar{\overline{x}}
\def\ybar{\overline{y}}
\def\zbar{\overline{z}}

\def\tp{{\rm tp}}

\def\tr/{trivial}
\def\nt/{non--trivial}
\def\st/{strong type}

\def\TV/{Tarski--Vaught}

\def\sc/{sound construction}
\def\ac/{atomic construction}

\def\fal/{functional}
\def\upl/{unique parallel lines}
\def\chp/{categorical in a higher power}

\def\includeE#1{{\lhook\kern-3.5pt\joinrel\smash{
    \mathop{\longrightarrow}\limits^{#1}}}}

\def\efor/{Example~\ref{E4}}

\def\BL/{Baldwin--Lachlan}
\def\Bu/{Buechler}
\def\Hr/{Hrushovski}
\def\lm/{locally modular}
\def\wm/{weakly minimal}
\def\nm/{non--modular}
\def\ss/{superstable}
\def\ud/{unidimensional}
\def\sm/{strongly minimal}

\def\abar{\overline{a}}

\def\bbar{\overline{b}}
\def\cbar{\overline{c}}
\def\dbar{\overline{d}}
\def\ebar{\overline{e}}
\def\fbar{\overline{f}}

\def\hbar{\overline{h}}

\def\kbar{\overline{k}}

\def\mbar{\overline{m}}

\def\xbar{\overline{x}}
\def\ybar{\overline{y}}
\def\zbar{\overline{z}}

\def\tp{{\rm tp}}

\def\tr/{trivial}
\def\nt/{non--trivial}
\def\st/{strong type}
\def\conc{{\char'136}}
\def\abar{\bar{a}}
\def\bbar{\bar{b}}
\def\cbar{\bar{c}}
\def\dbar{\bar{d}}
\def\ebar{\bar{e}}
\def\ybar{\bar{y}}
\def\phi{\varphi}

\def\C{{\mathfrak  C}}

\def\FF{{\bf F}}

\def\S{{\cal S}}

\def\Z{{\mathbb Z}}
\def\tp{{\rm tp}}

\def\Fa0{{\FF^a_{\aleph_0}}}

\def\<{\langle}
\def\>{\rangle}

\newtheorem{Theorem}{Theorem}[section]
\newtheorem{Proposition}[Theorem]{Proposition}
\newtheorem{Definition}[Theorem]{Definition}

\newtheorem{Remark}[Theorem]{Remark}
\newtheorem{Example}[Theorem]{Example}
\newtheorem{Lemma}[Theorem]{Lemma}
\newtheorem{Corollary}[Theorem]{Corollary}

\newtheorem{Fact}[Theorem]{Fact}

\def\ss{{\bf s}}

\DeclareMathOperator{\Aut}{Aut}

\def\K1{{\mathbf K_1}}

\newcommand\myrestriction{\mathord\restriction}
\def\mr#1{\myrestriction_{#1}}
\def\0bar{\overline{0}}

\begin{document}

	\title{Forking and invariant types in monadic NIP theories}
	
	\date{\today}

	\maketitle
	
	\begin{abstract}  \noindent We obtain a strong decomposition theorem for invariant global types in a monadically NIP theory.  From this, we prove that if an $n$-type $\tp(\abar/MC)$ does not fork over $M$  then $\abar=(\fbar,\dbar)$ where $\tp(\dbar/MC)$ is $M$-definable and $\tp(\fbar/MC\dbar)$ is finitely satisfied in $M$.  
Over arbitrary base sets $B$, we prove that an $n$-type $\tp(\abar/BC)$ does not fork over $B$ if and only if $\tp(a/BC)$ does not fork over $B$ for each singleton $a\in\abar$.  We show that monadically NIP theories
satisfy density of definable types among non-forking extensions.  
\end{abstract}

\section{Introduction}  

The notions of {\em forking} and {\em non-splitting} were introduced by Saharon Shelah over fifty years ago.  They were explored in depth for stable theories, but the notions make sense inside a monster model $\C$ of any complete theory $T$.   Not surprisingly, the properties of (non)-forking depend on the assumptions on the theory $T$.  Here, we are interested in describing when $$\fg {\abar} B C$$
holds,  i.e., `$\tp(\abar/BC)$ does not fork over $B$' where $B$ and $C$ are small subsets of $\C$ and $\abar$ is a finite tuple.
Here, we concentrate on {\em monadically NIP} theories (see Fact~\ref{monNIPequiv}). 

The analysis of forking over models is easier than over general sets, mostly due to 
Fact~\ref{start} below.   Understanding non-forking over models is related to an analysis of the $M$-invariant global types over the monster model $\C$.   The $M$-invariant global 1-types in a dp-minimal theory were classified by Pierre Simon, see Fact~\ref{dptrichotomy}.   
Section~\ref{PPP} gives many preparatory lemmas using weaker hypotheses, but the culmination of the section is Theorem~\ref{biginvariant}, which gives a strong decomposition theorem for an arbitrary global, $M$-invariant $n$-type in a monadically NIP theory.  

Next, in Section~\ref{QQQ}, we drop this down to types over small sets $MC$ that contain an elementary substructure $M\preceq\C$.  This is surprisingly delicate. One of the main issues is that a 1-type over $MC$  that is both $M$-definable and finitely satisfied in $M$ need not have the strong consequences of a global type with these properties.  (The 1-type of being infinitely large in a model of DLO is one example of this.)  Despite this, Theorem~\ref{partition} gives a  decomposition theorem for an $n$-tuple
$\tp(\abar/MC)$ that does not fork over $M$.    From this, we get a `1-sided triviality' result, Corollary~\ref{triviality} that is related to the fact that non-forking is totally trivial in any monadically stable theory.    The following restates Theorem~\ref{partition} and Corollary~\ref{triviality}.

\begin{Theorem} [$T$ monadically NIP]  \label{restate} Suppose $M\preceq\C$ is small, $C\supseteq M$ is small, and $\abar=(a_1,\dots,a_n)$ is a finite tuple.
Let $$F^*=\{i:\tp(a_i/C) \ \hbox{is finitely satisfied in $M\}$}\quad \hbox{and}  \quad D=[n]\setminus F^*$$
Then the following are equivalent.
\begin{enumerate}
\item  $\tp(\abar/C)$ does not fork over $M$;
\item  $\tp(a_i/C)$ does not fork over $M$ for every $i$; and
\item  $\tp(\abar_D/C)$ is $M$-definable and $\tp(\abar_{F^*}/C\abar_D)$ is finitely satisfied in $M$.
\end{enumerate}
\end{Theorem}

Next, in Section~\ref{ZZZ} we use the previous results to describe forking over arbitrary sets.  Our  main results are\footnote{After writing this article, we understand that Artem Chernikov has independently proved Theorem~\ref{restate2}.  However, the proofs are rather different.}

\begin{Theorem} [$T$ monadically NIP]  \label{restate2} Suppose $B,C$ are small sets and $\abar=(a_1,\dots,a_n)$ is a finite tuple, and $c$ is any singleton.  
\begin{enumerate}  
\item $\tp(\abar/BC)$ does not fork over $B$ if and only if $\tp(a_i/BC)$ does not fork over $B$ for each $a_i\in\abar$; and
\item  Suppose $\fg {\abar}  B C$.  Then either $\fg {\abar c} B C$ or $\fg {\abar} B {Cc}$ holds. 
\item  For any small set $B$ and singletons $a,b,c$, $\nfg a B b$ and $\nfg b B c$ imply $\nfg a B c$.
\end{enumerate}
\end{Theorem}

Finally, in Section~\ref{RRR} we use the decomposition given in Theorem~\ref{partition} to  prove that the set of $M$-definable global types 
are dense in the set of all non-forking global types over $M$.  This was proved for dp-minimal theories with property (D) by Pierre Simon and Sergei Starchenko in \cite{SS}, but we close with an example of a monadically NIP theory that fails property (D).


\medskip
\noindent

\section{General remarks about forking}

In this paper, we investigate the non-forking relation, first  over models $M$, but then over arbitrary sets $B$, in monadically NIP theories.    
In \cite{BL1,BL2} Braunfeld and the author gave many characterizations of a theory being monadically NIP.   
In this paper we use these three equivalents.  

\begin{Fact}  \label{monNIPequiv}  The following are equivalent for a complete theory $T$.
\begin{enumerate}
\item  $T$ is monadically NIP;
\item  $T$ is dp-minimal and has endless indiscernible triviality (EIT);
\item  $T$ has the f.s.\ dichotomy.
\end{enumerate}
\end{Fact}  

From \cite{BS}, $T$ has {\em monadic NIP} if the theory of the expansion of any model of $T$ by any number of unary predicates
is NIP.   Clearly, if $Th(M)$ is monadic NIP, then so is $Th(M,U)$, where $(M,U)$ is an arbitrary expansion of $M$ by a unary predicate.
From \cite{BL1,BL2}, $T$ has {\em endless indiscernible triviality (EIT)}  if, for every sequence $I=(\ebar_i:i\in\Z)$ indexed by the integers and for all sets $A,B$,
if $I$ is indiscernible over $A$ and indiscernible over $B$, then $I$ is indiscernible over $AB$.   A theory $T$ has the {\em f.s.\ dichotomy} if, for all small $M\preceq\C$, all small sets $A,B$, and all singletons $c$, if $\tp(A/MB)$ is finitely satisfied in $M$, then at least one of $\tp(Ac/MB)$ and $\tp(A/MBc)$ is finitely satisfied in $M$ as well.

The starting point for our discussion is the following result, which appears in many places, including 5.22 of \cite{PierreBook}.

\begin{Fact}   \label{start}   ($T$ NIP)   Suppose $M\preceq\C$ is small.  An $L(\C)$ formula  $\phi(\xbar,\dbar)$ does not fork over $M$ if and only if there is a global type
$p(\xbar)\in S(\C)$ that is $M$-invariant and $\phi(\xbar,\dbar)\in p$.
\end{Fact}

Thus, if we hope to understand forking over models in a NIP theory, we need to understand what the global $M$-invariant types are.   There are two natural examples of $M$-invariant global types in any complete theory.  The following appear as  Examples 2.16 and 2.17 of \cite{PierreBook}.  

\begin{Fact}   Suppose $p(\xbar)\in S(\C)$ is a global type and $A\subseteq\C$ is small.  If either $p$ is $A$-definable or if $p$ is finitely satisfied in $A$, then
$p(\xbar)$ is $A$-invariant.  
\end{Fact}

If we strengthen the theory to being dp-minimal and restrict only to $M$-invariant global 1-types, these are the only examples.
However, sometimes both of these conditions can occur simultaneously.  
Call an $M$-invariant global type $q\in\S(\C)$ {\em generically stable over $M$} if $q$ is both $M$-definable and is finitely satisfiable in $M$.   This notion has many equivalents, see e.g., 2.29 of \cite{PierreBook}, but it is important that we apply this adjective to global types.  In particular, any Morley sequence built from a generically stable global type over $M$ is totally indiscernible over $M$.  However, taking $T=DLO$, the `infinitely large' type $p(x)\in S_1(M)$ is both $M$-definable and (trivially)  is finitely satisfied in $M$.    But, for global types, we have the following trichotomy, which appears as Theorem~2.8 of \cite{dpInv}.

\begin{Fact} \label{dptrichotomy}   [$T$ dp-minimal]  Let $p(x)\in S_1(\C)$ be a global, $M$-invariant 1-type.  
Then one of the three possibilities holds.
\begin{itemize}
\item   $p$ is finitely satisfied in $M$ but not $M$-definable;
\item  $p$ is $M$-definable but not finitely satisfied in $M$; or
\item  $p$ is generically stable over $M$ (i.e., both finitely satisfied in $M$ and $M$-definable).
\end{itemize}
\end{Fact}

\section{Global $M$-invariant types in monadic NIP theories}  \label{PPP}

The goal of this section is the decomposition theorem, Theorem~\ref{biginvariant} for monadically NIP theories.  
Given that Fact~\ref{dptrichotomy} gives us a description of the global, $M$-invariant 1-types, we need methods for
`gluing together' compatible global $M$-invariant types.  Although the full strength of monadic NIP is 
used in the main results,  many of the preparatory results can be proved from weaker hypotheses.
To begin, it is remarkable that the following Proposition only requires endless indiscernible triviality.  In particular, NIP is not required.  

\begin{Proposition} [$T$ EIT]  \label{EITprop}  Suppose $M\preceq\C\preceq\C^+$ with $M$ small and $\C$ a monster model.
Suppose $p(\xbar),q(\ybar)\in S(\C)$ are both global, $M$-invariant types.  Then for every $\abar,\bbar$ from $\C^+$, if $\tp(\abar/\C)=p$ and $\tp(\bbar/C)=q$,
then $\tp(\abar\bbar/\C)$ is $M$-invariant as well.
\end{Proposition}

\begin{proof}  Checking that a global type is $M$-invariant is finitary.   Thus, choose any $\cbar,\cbar'$ from $\C$ such that $\tp(\cbar/M)=\tp(\cbar'/M)$.  It suffices to show
that $\tp(\cbar/M\abar\bbar)=\tp(\cbar'/M\abar\bbar)$.  Trivially, $\tp(\cbar/M)$ is finitely satisfied in $M$, so let  $r(\zbar)\in S(\C)$ be a global type extending $\tp(\cbar/M)$ that
is finitely satisfied in $M$.  Choose any $\dbar$ from $\C$ realizing the restriction $r|M\cbar\cbar'$.  
We will show that $\tp(\cbar/M\abar\bbar)=\tp(\dbar/M\abar\bbar)=\tp(\cbar'/M\abar\bbar)$, which, as noted above, suffices.
The argument for each of these equalities is symmetric, so we concentrate on the first one.   
Note that the two-element sequence $\cbar,\dbar$ is the beginning of a Morley sequence in $r(\zbar)$ over $M$.   Thus, 
by compactness and homogeneity, the Morley pair $(\cbar,\dbar)$ can be included in  an $M$-indiscernible sequence
$(\ebar_i:i\in\Z)$ from $\C$, indexed by $\Z$.

\medskip
\noindent{\bf Claim.}   $(\ebar_i:i\in\Z)$ is indiscernible over both $M\abar$ and $M\bbar$.

\begin{proof}  The arguments are symmetric, so we concentrate on the first.   Choose any $i_1<\dots<i_n$ and $j_1<\dots<j_n$ from $\Z$.
As $(\ebar_i:i\in\Z)$ is $M$-indiscernible, 
$$\tp(\ebar_{i_1},\dots,\ebar_{i_n}/M)=\tp(\ebar_{j_1},\dots,\ebar_{j_n}/M)$$
Thus, since $\tp(\abar/\C)$ is $M$-invariant, we conclude 
$$\tp(\ebar_{i_1},\dots,\ebar_{i_n}/M\abar)=\tp(\ebar_{j_1},\dots,\ebar_{j_n}/M\abar)$$
completing the proof of the Claim.
\qed
\end{proof}

Thus, by endless indiscernible triviality, we have $(\ebar_i:i\in\Z)$ indiscernible over $M\abar\bbar$.  As $\cbar,\dbar$ are contained in the sequence, it follows
that  $\tp(\cbar/M\abar\bbar)=\tp(\dbar/M\abar\bbar)$, as required.

The verification of the second equality is symmetric, by choosing a separate $M$-indiscernible sequence $(\ebar_i':i\in\Z)$ containing both $\cbar'$ and $\dbar$
and arguing as above.
\qed
\end{proof}

Thus, by induction on $n$ we obtain the following:

\begin{Corollary}   \label{iterateMinv}  [$T$  EIT]  Suppose $M\preceq\C$ is small.
For any $n$, a global type $p(x_1,\dots,x_n)\in S_n(\C)$ is $M$-invariant if and only if $p\mr{x_i}(x_i)\in S_1(\C)$ is $M$-invariant for each $i$.
\end{Corollary}

Because of this corollary, it is natural to investigate global 1-types.  
%
%
The following characterization of a global type being $M$-definable will be heavily used in what follows.

\begin{Fact}   \label{charMdef}  [$T$ any complete theory]   Suppose $M\preceq\C$.   The following are equivalent for a global type $p(\xbar)\in S_{\xbar}(\C)$.
\begin{enumerate}
\item $p$ is $M$-definable;
\item  $p$ is the unique heir of $p\mr{M}$;
\item  $p$ is $M$-invariant and  $p(\xbar)\otimes q(\ybar)=q(\ybar)\otimes p(\xbar)$ for all global types $q(\ybar)\in S_{\ybar}(\C)$ that are finitely satisfied in $M$.
\end{enumerate}
\end{Fact}

\begin{proof}  The equivalence $(1)\Leftrightarrow(2)$ is Theorem~4.2 of \cite{LP}.   $(1)\Rightarrow (3)$ appears as Lemma~2.23 of \cite{PierreBook}, and 
$(3)\Rightarrow (1)$, which appeals to  $(1)\Leftrightarrow(2)$ is proved in Lemma~2.3 of \cite{dpInv}.  
\qed
\end{proof}  

\begin{Definition}  {\em  A small set $C\supseteq M$ is {\em full for $M$} if, for all $n$, every $p(\xbar)\in S_n(M)$ is realized in $C$.
}
\end{Definition}  

The utility of fullness is partially explained by the following fact, which holds for any complete theory.

\begin{Fact}  \label{fullstationary}  [$T$ any complete theory]  Suppose $C\supseteq M$ is full for $M$.
\begin{itemize}
\item  If 
$p=\tp(\abar/C)$ is finitely satisfiable in $M$ then for any $D\supseteq C$, if $q,r\in S(D)$ are both extensions of $p$ that are finitely satisfied in $M$,
then $q=r$.    
\item  If $p,q\in S(\C)$ are global $M$-invariant types, then $p\mr{C}=q\mr{C}$ implies $p=q$.  Moreover,
$p$ is finitely satisfied in $M$ if and only if $p\mr{C}$ is finitely satisfied in $M$.
\end{itemize}
\end{Fact}

If we assume $T$ is monadically NIP we can say much more.  The following Fact appears as Lemma 3.3(2,3) of \cite{BL1}.

\begin{Fact} \label{3.3}   [$T$ monadically NIP]    
\begin{enumerate}
\item  Suppose $C\supseteq M$ is arbitrary.  If $\tp(a/C)$ is finitely satisfied in $M$ for every $a\in A$, then $\tp(A/C)$ is finitely satisfied in $M$.
\item  
Suppose $C\supseteq M$ is full for $M$. Then for any  non-empty sets $A,B$ $\tp(A/BC)$ is finitely satisfied in $M$ if and only if $\tp(a/Cb)$ is finitely
satisfied in $M$ for every $a\in A$ and $b\in B$.
\end{enumerate}
\end{Fact}

We combine the previous results assuming $T$ is dp-minimal and has EIT, but these two properties characterize being monadically NIP by Fact~\ref{monNIPequiv}.    Whereas Proposition~\ref{meld}(1) was already known in \cite{BL1}, clause (2) is arguably the most important new ingredient of this paper.

\begin{Proposition}  \label{meld} [$T$ monadically NIP]    
Suppose $M\preceq\C$ is small and let $p(\xbar)\in S_{\xbar}(\C)$ be an $M$-invariant global type.  Let $\xbar=\xbar_1\conc\xbar_2$ be any proper partition
and let $p_1(\xbar_1)=p\mr{\xbar_1}$ and $p_2(\xbar_2)=p\mr{\xbar_2}$ be the restrictions.  
\begin{enumerate}
\item  If both $p_1$ and $p_2$ are finitely satisfied in $M$, then $p$ is finitely satisfied in $M$.
\item  If both $p_1$ and $p_2$ are $M$-definable, then $p$ is $M$-definable.
\item  If both $p_1$ and $p_2$ are  generically stable over $M$, then $p$ is generically stable over $M$.
\end{enumerate}
\end{Proposition}  

\begin{proof}  (1)  Assume both $p_1$ and $p_2$ are finitely satisfied in $M$ and choose a small set $C\supseteq M$ that is full for $M$.  
It suffices to show that $\tp(\abar/C)$ is finitely satisfied in $M$ for every/some realization of $p|C$.  Let $\abar=(\abar_1,\abar_2)$ realize $p|C$, where $\tp(\abar_\ell/C)=p_\ell|C$ for $\ell=1,2$.  As both $p_1$ and $p_2$ are finitely satisfied in $M$, $\tp(a/C)$ is finitely satisfied in $M$ for every $a\in\abar$, hence $\tp(\abar/C)$ is finitely satisfied in $M$ by Fact~\ref{3.3}(1).

(2)  We show that $p(\xbar)$ is $M$-definable by employing Fact~\ref{charMdef}(3).  
Choose any global type $q(\ybar)\in S_{\ybar}(\C)$ that is finitely satisfied in $M$.  Towards showing that  $p(\xbar)\otimes q(\ybar)=q(\ybar)\otimes p(\xbar)$, 
choose a full $C\supseteq M$ and let $(\abar,\bbar)$ realize $(p(\xbar)\otimes q(\ybar))|C$.   Showing that $(\abar,\bbar)$ realizes $(q(\ybar)\otimes p(\xbar))|C$ amounts
to showing that $\tp(\bbar/C\abar)=q|C\abar$.    Clearly, $\tp(\bbar/C)=q|C$, so since $C\supseteq M$ is full, by Fact~\ref{fullstationary} it suffices to show that $\tp(\bbar/C\abar)$ is finitely satisfied in $M$.  
Write $\abar=\abar_1\conc\abar_2$, where the tuple $\abar_\ell$ realizes $p_\ell(\xbar_\ell)$
for $\ell=1,2$.  Since $p_\ell(\xbar_\ell)$ is $M$-definable, $p_\ell(\xbar_\ell)\otimes q(\ybar)=q(\ybar)\otimes p_\ell(\xbar_\ell)$ by Fact~\ref{charMdef}(3).   
Thus, $\bbar$ realizes $q|C\abar_\ell$.  In particular, both  $\tp(\bbar/C\abar_1)$ and $\tp(\bbar/C\abar_2)$ are finitely satisfied in $M$.  It follows that $\tp(\bbar/Ca)$ is finitely
satisfied in $M$ for every $a\in \abar$.  Thus, $\tp(\bbar/C\abar)$ is finitely satisfied in $M$ by Fact~\ref{3.3}(2),  As noted above, this implies $\bbar$ realizes $q|C\abar$,
as required.

(3)  This follows immediately from (1) and (2).
\qed
\end{proof}

\begin{Corollary}   \label{coordinates}  [$T$ monadic NIP]   Let $p(x_1,\dots,x_n)$ be a global, $M$-invariant type and, for each $i$, let $p_i=p\mr{x_i}$ be the restriction of $p$ to formulas whose only free variable is $x_i$.  
Then:
\begin{enumerate}
\item  $p(x_1,\dots,x_n)$ is finitely satisfied in $M$ if and only if each $p_i(x_i)$ is finitely satisfied in $M$.
\item  $p(x_1,\dots,x_n)$ is  $M$-definable if and only if each $p_i(x_i)$ is $M$-definable.
\item  $p(x_1,\dots,x_n)$ is generically stable over $M$ if and only if each $p_i(x_i)$ is generically stable over $M$.
\end{enumerate}
\end{Corollary}

Now, in light of Corollary~\ref{iterateMinv} these are not all of the $M$-invariant global types.   In particular, there can be `hybrids' among these three species of $M$-invariant global 1-types.   Theorem~\ref{biginvariant} below shows that we can decompose any $M$-invariant $n$-type into its `components.'

\begin{Definition}  {\em  We say global types $q(\xbar),r(\ybar)$ are {\em orthogonal}, $q\perp r$, if $q(\xbar)\cup r(\ybar)$ generates a complete type in $S_{\xbar,\ybar}(\C)$.    Say three global types $q(\xbar),r(\ybar),t(\zbar)$
are {\em triply orthogonal} if $q(\xbar)\cup r(\ybar)\cup t(\zbar)$ generates a complete type in $S_{\xbar,\ybar,\zbar}(\C)$.
}
\end{Definition}

That is, triple orthogonality means there is only one choice for $\tp(\abar\bbar\cbar/\C)$ whenever
$\abar,\bbar,\cbar$ realize $q,r,t$, respectively.

\begin{Lemma}  \label{orthlemma}  [$T$ monadic NIP]  Suppose $q(\xbar), r(\ybar)$ are $M$-invariant with $r(\ybar)$ finitely satisfied in $M$. Let $\C^+\succeq\C$ be $|\C|^+$-saturated.  
\begin{enumerate}
\item   Suppose $\abar,\bbar$ are from $\C^+$, $\abar$ realizes $q(\xbar)$, $\bbar$ realizes $r(\ybar)$, and $\tp(\bbar/\C\abar)$ is finitely satisfied in $M$.  
Then $\tp(\abar\bbar/\C)=r(\ybar)\otimes q(\xbar)$.
\item  $q\perp r$ if and only if $\tp(\bbar/\C\abar)$ is finitely satisfied in $M$ for every $\abar,\bbar$ from $\C^+$ with $\abar$
realizing $q(\xbar)$ and $\bbar$ realizing $r(\ybar)$.
\item  $q\perp r$ if and only if $q\perp r\mr{y_i}$ for every $y_i\in\ybar$.
\end{enumerate}
\end{Lemma}

\begin{proof}  (1)  It suffices to show that $\tp(\bbar/\C\abar)=r|\C\abar$, the canonical invariant extension of $r$.  Both sides extend $r$ and are finitely satisfied in $M$, so equality follows  from Fact~\ref{fullstationary} since $\C$
is full  for $M$.  

(2)  First, suppose $q\perp r$ and choose any $\abar,\bbar$ as in the hypotheses.  
One possibility for $\tp(\abar\bbar/\C)$ is $r(\ybar)\otimes q(\xbar)$,
so we must have $\tp(\abar\bbar/\C)=r(\ybar)\otimes q(\xbar)$ by the orthogonality.  Unpacking the definition gives $\tp(\bbar/\C\abar)=r|\C\abar$.    But, as $r$ is finitely satisfied in $M$, so is $\tp(\bbar/\C\abar)$.
 
 Conversely, suppose $\abar$ is any realization of $q(\xbar)$ and $\bbar$ is any realization of $r(\ybar)$.  From our assumption, (1) implies that $\tp(\abar\bbar/\C)=r(\ybar)\otimes q(\xbar)$.  
 
 (3)  Left to right is trivial, so assume $q\perp r\mr{y_i}$ for every $y_i\in\ybar$.  
Choose any $\abar$ realizing $q(\xbar)$ and any $\bbar$ realizing $r$.   As $q\perp r\mr{y_i}$ for every $y_i\in \ybar$, it follows from (2) that $\tp(b_i/\C\abar)$ is
finitely satisfied in $M$ for every $b_i\in \bbar$.  Thus,   $\tp(\bbar/\C\abar)$ is finitely satisfied in $M$ by Fact~\ref{3.3}(1).  
Thus, $q\perp r$ by (2).
\qed
\end{proof}

\begin{Proposition}  \label{orthprop}  [$T$ monadically NIP]  Suppose $q(\xbar),r(\ybar)$ are $M$-invariant with $r(\ybar)$ finitely satisfied in $M$.
\begin{enumerate}
\item  If $q(\xbar)$ is $M$-definable, but no $r\mr{y_i}$ is generically stable over $M$, then $q\perp r$.
\item  If $q\mr{x_i}$ is not finitely satisfiable in $M$ for any $x_i\in\xbar$, then $q\perp r$.
\end{enumerate}
\end{Proposition}

\noindent The two proofs are rather different.

\begin{proof}  (1)  In light of Lemma~\ref{orthlemma}(3), we may assume $\lg(\ybar)=1$, i.e., $r(y)$ is a global 1-type, finitely satisfiable in $M$, but not generically stable over $M$.  
By 9.13 of \cite{PierreBook}, $r(y)$ is distal.   As $q(\xbar)$ is  $M$-definable, by
Fact~\ref{charMdef}(3) $q$ commutes with every finitely satisfied global type over $M$.  In particular, $q$ and $r$ commute.  Thus, $q\perp r$ by 9.11 of \cite{PierreBook}.  

(2)  Choose (in any elementary extension of $\C$) $\abar$ realizing $q(\xbar)$ and $\bbar$ realizing $r(\ybar)$.   

\medskip
\noindent{\bf Claim.}  $\tp(\bbar/\C\abar)$ is finitely satisfied in $M$.

\begin{proof}  Write $\abar=(a_i:i<n)$ and, for each $i<n$, let $\abar_{<i}=(a_j:j<i)$.  
We argue by induction on $i\le n$ that $\tp(\bbar/\C\abar_{<i})$ is finitely satisfied in $M$, which suffices to prove the Claim.
By assumption $\tp(\bbar/\C)=r$, hence it is finitely satisfied in $M$, so this holds for $i=0$.  Now fix $i<n$
and assume $\tp(\bbar/\C\abar_{<i})$ is finitely satisfied in $M$.   We apply the f.s.\ dichotomy to this type and the element $a_i$.
By our assumption on $q$, $\tp(a_i/\C)$ is not finitely satisfied in $M$, hence neither is $\tp(\bbar a_i/\C\abar_{<i})$.   Thus, the f.s.\ dichotomy implies
$\tp(\bbar/\C\abar_{<i} a_i)$ is finitely satisfied in $M$, completing the inductive step.
\qed
\end{proof}

Given the Claim, $q\perp r$ follows immediately from Lemma~\ref{orthlemma}(2).
\qed
\end{proof}


\begin{Theorem} \label{biginvariant}   [$T$ monadic NIP]   Let $p(x_1,\dots,x_n)$ be any $M$-invariant global type.   Partition $[n]=F\sqcup G\sqcup D$ (some of these might be empty)
where
\begin{itemize}
\item  $F=\{i\in [n]:p_i$ is finitely satisfied in $M$, but is not $M$-definable$\}$;
\item  $G=\{i\in[n]:p_i$ is generically stable over $M\}$; and
\item $D=\{i\in[n]:p_i$ is $M$-definable but not finitely satisfied in $M\}$.
\end{itemize}
Let $p_F(\xbar_F)$, $p_G(\xbar_G)$, and $p_D(\xbar_D)$ be the restrictions of $p$ to each of the three subsequences $\xbar_F$, $\xbar_G$, $\xbar_D$ of $\xbar$.
Then 
\begin{enumerate}
\item  $p_F(\xbar_F)$ is finitely satisfied in $M$ but is not $M$-definable if $F\neq\emptyset$;
\item  $p_G(\xbar_G)$ is generically stable over $M$; and
\item  $p_D(\xbar_D)$ is $M$-definable but not finitely satisfied in $M$ if $D\neq\emptyset$.
\item  The three types $p_F$, $p_G$, $p_D$ are triply orthogonal, i.e.,  $p_F\cup p_G\cup p_D\models p$.
\end{enumerate}
\end{Theorem}

\begin{proof}  (1), (2), (3) follow immediately from Corollary~\ref{coordinates}.  
(4)  Choose any realizations (in any elementary extension of $\C$)  $\abar,\bbar,\cbar$ of $p_F,p_G,p_D$, respectively.
By Proposition~\ref{orthprop}(1), $p_F\perp p_G$, so $(\abar,\bbar)$ must realize $p_F(\xbar_F)\otimes p_G(\xbar_G)$.  Let $r(\xbar_F,\xbar_G)$ denote this product type.  The type $r(\xbar,\ybar)$ is finitely satisfied in $M$ as both factors are.
Then by Proposition~\ref{orthprop}(2) applied to $p_D(\xbar_D)$ and $r(\xbar_F,\xbar_G)$, we get that
$(\abar,\bbar,\cbar)$ realizes $r(\xbar_F,\xbar_G)\otimes p_D(\xbar_D)$.  Thus, by associativity, $\tp(\abar\bbar\cbar/\C)=p_F\otimes p_G\otimes p_D$.
Finally, let $\dbar$ be any realization of $p$.   Partition  $\dbar=\abar\bbar\cbar$ according to the partition of $[n]$ described above.  
Then from above, $\tp(\dbar/\C)=p_F\otimes p_G\otimes p_D$.  As $\dbar$ is an arbitrary realization of $p$, $p=p_F(\xbar_F)\otimes p_G(\xbar_G)\otimes p_D(\xbar_D)$.
\qed
\end{proof}

\section{From global types to non-forking over small models}  \label{QQQ}   

\medskip\centerline{{\bf Throughout this section, $T$ is monadically NIP, $M\preceq\C$ is small and $C\subseteq\C$ is small.}}

\medskip

The classification above is fine, but we need to drop this down to analyze $\tp(\abar/MC)$ when it does not fork over $M$.   The complication is that {\em a priori,} such a type may have many different $M$-invariant global extensions.   Because of this, the orthogonality results described in the section above do not directly apply to types over $MC$.

\begin{Lemma}  \label{overM}  Suppose $\abar=(a_1,\dots,a_n)$ and $\tp(a_i/M)$ is $M$-definable for each $i$.  Then $\tp(\abar/M)$ is also $M$-definable.
\end{Lemma}

\begin{proof}  Let $r(\xbar)=\tp(\abar/M)$ and let $R(\xbar)\in S(\C)$ be any global heir extension of $r$.   Note that each restriction $R\mr{x_i}$ is an heir extension of $\tp(a_i/M)$, which we assumed to be $M$-definable.   Thus, by Fact~\ref{charMdef}, $R\mr{x_i}$ must be the unique $M$-definable extension of $\tp(a_i/M)$ to $S_1(\C)$.   In particular, 
each $R\mr{x_i}$ is $M$-invariant.  Thus, by Corollary~\ref{iterateMinv}, $R(x_1,\dots,x_n)$ is also $M$-invariant.  It follows from Corollary~\ref{coordinates}(2) that $R(x_1,\dots,x_n)$ is $M$-definable as well.  But, as $R$ was chosen to extend $\tp(\abar/M)$, the latter is $M$-definable as well.
\qed
\end{proof}

It is somewhat surprising that the following Lemma requires an argument, as opposed to following directly from Lemma~\ref{overM}.  The issue is that, even in monadically NIP theories, there can be many different extensions of an $M$-definable $p\in S(M)$ to a larger base set.   (One example is DLO, where $p(x)\in S(M)$ is the type of being `infinitely large.')  Before stating the Lemma, we note an easy fact, which is proved simply by composing the definitions.

\begin{Fact}  \label{compose}   [$T$ any complete theory]  Suppose $M\preceq\C$,  and $C,\abar,\bbar$ are from $\C$.  If both $\tp(\bbar/MC)$ and $\tp(\abar/MC\bbar)$ are $M$-definable, then $\tp(\abar\bbar/MC)$ is $M$-definable as well.
\end{Fact}  

\begin{proof}  Choose any $\phi(\xbar,\ybar,\zbar)$.  By the $M$-definability of $\tp(\abar/MC\bbar)$, there is $\theta(\ybar,\zbar)\in L(M)$ such that, for every $\cbar\in MC$, $$\phi(\abar,\bbar,\cbar)\quad\Longleftrightarrow\quad \theta(\bbar,\cbar).$$
By the $M$-definability of $\tp(\bbar/MC)$ there is $\delta(\zbar)$ such that $\theta(\bbar,\cbar)\Leftrightarrow\delta(\cbar)$.
This $\delta(\zbar)\in L(M)$ is an $M$-definition of $\phi(\xbar,\ybar;\zbar)$ in $\tp(\abar\bbar/MC)$.
\end{proof}

\begin{Lemma}  \label{overMC}  Suppose $\abar=(a_1,\dots,a_n)$ and $\tp(a_i/MC)$ is $M$-definable for each $i$.  Then $\tp(\abar/MC)$ is also $M$-definable.
\end{Lemma}

\begin{proof}  First, note that our assumption implies $\tp(a_i/M)$ is $M$-definable for each $i$, hence $\tp(\abar/M)$ is $M$-definable by Lemma~\ref{overM}.
We argue that $\tp(\abar/MC)$ is also $M$-definable by induction on $n$.  For $n=1$ there is nothing to prove.  So assume $\ebar=(e_1,\dots,e_n)$ and a singleton $d$ are given with both  $\tp(\ebar/MC)$ and
$\tp(d/MC)$ $M$-definable.  
Being $M$-definable, $\tp(\ebar/MC)$ is an heir of $\tp(\ebar/M)$, so by  heir-coheir duality, $\tp(C/M\ebar)$ is finitely satisfied in $M$.   We apply the f.s.\ dichotomy to this with the element $d$.   There are two cases.

\medskip
\noindent{\bf Case 1.}  $\tp(C/M\ebar d)$ is finitely satisfied in $M$.
\medskip

In this case, by  heir-coheir duality, $\tp(\ebar d/MC)$ is an heir of $\tp(\ebar d/M)$.   But, as $\tp(\ebar d/M)$ is $M$-definable, it follows from Fact~\ref{charMdef} that this type has a unique heir to $MC$.   Thus, $\tp(\ebar d/MC)$ is this $M$-definable type.

\medskip
\noindent{\bf Case 2.}  $\tp(Cd/M\ebar)$ is finitely satisfied in $M$.
\medskip

Here, by  heir-coheir duality, $\tp(\ebar/MCd)$ is an heir of $\tp(\ebar/M)$.   As $\tp(\ebar/M)$ is $M$-definable, there is only one such heir, namely the $M$-definable one.
Thus $\tp(\ebar/MCd)$ is $M$-definable.  But also, $\tp(d/MC)$ is $M$-definable by assumption, hence $\tp(\ebar d/MC)$ is also $M$-definable by Fact~\ref{compose}.
\qed
\end{proof}

With this in hand, we can now prove a strong decomposition result for $\tp(\abar/MC)$ when $\tp(a_i/MC)$ does not fork over $M$ for every $a_i\in\abar$.  
In what follows, there is no reason why this partition must be proper, i.e., either $F^*$ or $D$ could be empty.  

\begin{Theorem}  \label{partition}
Suppose $\abar=(a_1,\dots,a_n)$ and assume $\tp(a_i/MC)$ does not fork over $M$ for every $i$.  Then there is a partition $[n]=F^*\sqcup D$ so that the induced partition
$\abar=(\fbar,\dbar)$ of $\abar$ satisfies $\tp(\dbar/MC)$ is $M$-definable and $\tp(\fbar/MC\dbar)$ is finitely satisfied in $M$.
\end{Theorem}   

\begin{proof}  Let $F^*=\{i\in [n]: \tp(a_i/MC)$ is finitely satisfied in $M\}$ and let $D=[n]\setminus F^*$.    Let $\fbar$ (resp.\ $\dbar$) be the subsequence of $\abar$ induced by $F^*$ (resp.\ $D$).  

Note that if $i\in D$, then since $p_i:=\tp(a_i/MC)$ does not fork over $M$, there is a global 1-type $P_i(x_i)\in S_1(\C)$  extending $p_i$ that is $M$-invariant.  By Fact~\ref{dptrichotomy} $P_i$ is either $M$-definable or else finitely satisfied in $M$.  However, if $P_i$ were finitely satisfied in $M$, then $p_i$ would be as well, contradicting $i\not\in F^*$.  Thus, $P_i$ and hence
$p_i$ is $M$-definable.  Thus, by Lemma~\ref{overMC}, $\tp(\dbar/MC)$ is $M$-definable.  

Next, we know that each singleton $f_i\in\fbar$ has $\tp(f_i/MC)$ finitely satisfied in $M$, so by Fact~\ref{3.3}(1), $\tp(\fbar/MC)$ is finitely satisfied in $M$ as well.
To finish, we must prove that $\tp(\fbar/MC\dbar)$ is finitely satisfied in $M$, which we prove by induction on $\lg(\dbar)$ (so if $D=\emptyset$ there is nothing to prove).
The argument is akin to the proof of the Claim in Proposition~\ref{orthprop}(2).
 Write $\dbar=(d_i:i<k)$ and we argue by induction on $i\le k$ that $\tp(\fbar/MC\dbar_{<i})$ is finitely satisfied in $M$.
To see this, assume $i<k$ and $\tp(\fbar/MC\dbar_{<i})$ is finitely satisfied in $M$.  We apply the f.s.\ dichotomy to this and the element $d_i$.  The point is that
if $\tp(\fbar d_i/MC\dbar_{<i})$ were finitely satisfied in $M$, then $\tp(d_i/MC)$ would be finitely satisfied in $M$, contradicting our choice of $d_i$.
Thus, we must have $\tp(\fbar/MC\dbar_{\le i})$ finitely satisfied in $M$, completing our inductive step.
\qed
\end{proof}

From the result above, the following Corollary is simply an instance of transitivity of non-forking.  

\begin{Corollary}  \label{triviality}   [$T$ monadically NIP]   If $\abar=(a_1,\dots,a_n)$ and $\tp(a_i/MC)$ does not fork over $M$ for every $i$, then
$\tp(\abar/MC)$ does not fork over $M$.
\end{Corollary}

\begin{proof}  By Theorem~\ref{partition}  choose a decomposition $\abar=(\fbar,\dbar)$ with $\tp(\dbar/MC)$ $M$-definable and $\tp(\fbar/MC\dbar)$ finitely satisfied in $M$.
Thus, $\tp(\dbar/MC)$ does not fork over $M$ and $\tp(\fbar/MC\dbar)$ does not fork over $M$, so $\tp(\fbar\dbar/MC)$ does not fork over $M$ by transitivity of non-forking, see e.g., 5.18 of \cite{PierreBook}.  
\qed
\end{proof}

\section{Non-forking over arbitrary bases}  \label{ZZZ}  

\medskip\centerline{{\bf Throughout this section, $T$ is assumed to be monadically NIP.}}
\medskip

For the whole of this section,  assume that $M\preceq\C\preceq\C^*$, where $M$ is small in a monster model $\C$, and $\C^*$ is a "monster for $\C$" i.e., $\C$ is small, relative to $\C^*$.   Note that if $p(\xbar)\in S(\C)$ is $M$-invariant, then there is a unique $M$-invariant extension $p^*\in S(\C^*)$ of $p$.  

\medskip

In this section, we extend some of the results from Section~\ref{QQQ}, most notably Corollary~\ref{triviality}, to instances of non-forking over arbitrary small  bases instead of small  models.  
This extension   begins with a standard fact, see e.g., 5.22 of \cite{PierreBook} and the discussion immediately thereafter.  

\begin{Fact}  \label{charforkinv}  Suppose $B\subseteq \C$ is any small set.   A global type $r(\xbar)\in S(\C)$ does not fork over $B$ if and only if $r$ is $M$-invariant for every small $M$ with $B\subseteq M\preceq\C$.  
\end{Fact}

\begin{Corollary}  \label{fuse}
For $\abar,\bbar\in\C^*$, if $\tp(\abar/\C)$ and $\tp(\bbar/\C)$ do not fork over $B$, then neither does $\tp(\abar\bbar/\C)$.
\end{Corollary}

\begin{proof}  Choose any small $M\supseteq B$.   By Fact~\ref{charforkinv}, $\tp(\abar/\C)$ and $\tp(\bbar/\C)$ are both $M$-invariant, hence so is $\tp(\abar\bbar/\C)$ by Corollary~\ref{iterateMinv}.  As $M$ is arbitrary, Fact~\ref{charforkinv} gives $\tp(\abar\bbar/\C)$ non-forking over $B$.
\end{proof}

%
\begin{Proposition} \label{MI}  Suppose $M\preceq\C\preceq \C^*$ with $M$ small and $\C^*$ a monster over $\C$.
Suppose $p(x),q(y)\in S_1(\C)$ with $p(x)$ $M$-invariant, but $q(y)$ not $M$-invariant.  Then for any $a,c\in\C^*$ with $a$ realizing $p$ and $c$ realizing $q$, we have $a$ realizes $p^*|\C c$.
\end{Proposition}

\begin{proof}  There are two cases, depending on the species of $M$-invariant type $p(x)$ is.

\medskip
\noindent{\bf Case 1.}  $p$ is finitely satisfied in $M$.  
\medskip

In this case, we apply the f.s.\ dichotomy directly.  We know that $\tp(a/\C)$ is finitely satisfied in $M$.  But, $\tp(ac/\C)$ cannot be finitely satisfied in $M$, since this would imply $\tp(c/\C)$ finitely satisfied in $M$, hence $M$-invariant.
Thus, by the f.s.\ dichotomy, $\tp(a/\C c)$ is finitely satisfied in $M$.   As $\C$ is full for $M$,  Fact~\ref{fullstationary} gives
$a$ realizes $p^*|\C c$, as required.

\medskip
\noindent{\bf Case 2.}  $p$ is  $M$-definable.
\medskip

First, since $q(y)$ is not $M$-invariant, there is a sequence $I=(\ebar_i:i\in \omega)$ that is $M$-indiscernible, but not $Mc$-indiscernible.  The argument is similar to part of the proof of Proposition~\ref{EITprop}.  As $q(y)$ is not $M$-invariant, choose tuples $\bbar,\bbar'$ from $\C$ with $\tp(\bbar/M)=\tp(\bbar'/M)$, but $\tp(\bbar/Mc)\neq\tp(\bbar'/Mc)$.    Let $r(\zbar)=\tp(\bbar/M)$.  As $r$ is (trivially)
finitely satisfied in $M$, choose a global $r^*(\zbar)\in S(\C)$ extending $r$ and finitely satisfied in $M$.  In $\C$, choose
a realization $\dbar$ of $r^*|M\bbar\bbar'$.  From the above inequality, we must have either $\tp(\bbar/Mc)\neq \tp(\dbar/Mc)$ or
$\tp(\bbar'/Mc)\neq\tp(\dbar/Mc)$.   As the arguments are symmetric, assume the former.  Then $(\bbar,\dbar)$ starts a Morley sequence in $r^*$ over $M$.   By compactness, there is an infinite, $M$-indiscernible sequence $I=(\ebar_i:i\in \omega)\subseteq\C$ with $\bbar,\dbar$ occurring in $I$.  The presence of $\bbar,\dbar\in I$ implies that $I$ is not $Mc$-indiscernible.  

Next, as $p$ is $M$-definable, so is the global extension $p^*(x)\in S(\C^*)$.  By way of contradiction, assume $a$ does not realize $p^*|\C c$.   
Then there is some $L$-formula $\phi(x,y,\zbar)$, whose definition is an $L(M)$-formula $\theta(y,\zbar)$, and $\hbar$ from $\C$  such that
$$\C^*\models \neg(\phi(a,c,\hbar)\leftrightarrow \theta(c,\hbar))$$
Choose a sequence $J=(a_j:j\in\omega)\subseteq\C^*$ where $a_0=a$ and $a_j$ realizes $p^*|\C c \abar_{<j}$ for each $j>0$.  Note that by forgetting $c$, $J$ is a Morley sequence in $p^*$ over $\C$, hence $J$ is indiscernible over $\C$.   Also, by
the properties of being a $p^*$-definition, $$\C^*\models \phi(a_j,c,\kbar)\leftrightarrow \theta(c,\kbar)\quad \hbox{for all $j>0, \kbar\in\C^{\lg(\zbar)}$}$$
As a technique for obviating the role of the parameter $\hbar$, we pass to an expansion $\C^*_U=(\C^*,U)$ of $\C^*$ by a unary predicate interpreted as $U^{\C^*_U}=\C$.  Let $L(U)=L\cup\{U\}$.   As $Th(\C^*)$ is monadic NIP, so is $Th(\C^*_U)$.  We will get a contradiction by showing that the two sequences $I$ and $J$ defined above are mutually indiscernible over $M$ in the language $L(U)$, but neither $I$ nor $J$ is indiscernible over $Mc$.  This configuration directly contradicts that $Th(\C^*_U)$ is dp-minimal.  

To see the  point of expanding the language, look at the $L(U)$-formula (with parameters from $M$)
$$\psi(x,y):=\exists \zbar\left(\bigwedge_{z_i\in\zbar} U(z_i)\wedge \neg[\phi(x,y,\zbar)\leftrightarrow \theta(y,\zbar)]\right)$$
Then $\C^*_U\models\psi(a_0,c)\wedge\neg \psi(a_j,c)$ for any $j>0$, so $J$ is not $Mc$-indiscernible in $L(U)$.  
 
It remains to show that $J$ is $MI$-indiscernible in $\C^*_U$ and $I$ is $MJ$-indiscernible in $\C^*_U$.  The former is easy.  From above, $J$ is $\C$-indiscernible in $\C^*$.  However, any $\sigma\in  Aut(\C^*)$ fixing $\C$ pointwise is an automorphism of the expansion $\C^*_U$.  Thus, $J$ is $\C$-indiscernible (and hence $MI$-indiscernible) in $\C^*_U$.
Thus, we are done once we prove the following Claim.

\medskip
\noindent{\bf Claim.}  $I$ is $MJ$-indiscernible in the expansion $\C^*_U$.  

\begin{proof}    Choose increasing sequences $i_1<\dots <j_n<\omega$ and $k_1<\dots<k_n<\omega$.  We begin by working in the original language $L$.  From above we know there is an automorphism $\sigma_0\in\Aut(\C)$ fixing $M$ pointwise with
$\sigma_0(\ebar_{i_j})=\ebar_{k_j}$ for each $1\le j\le n$.  Since $\tp(a_j/\C)=p$ is $M$-invariant for each $j\in\omega$,
$\tp(J/\C)$ is $M$-invariant by Corollary~\ref{iterateMinv}. 
Because   $\sigma_0$ permutes $\C$ while fixing $M$ pointwise, $\sigma_0\cup id_{J}$ is a 
partial $L$-elementary map in $\C^*$.
Since $\C^*$ is a monster model for $\C$ (in the original language)
there is an $L$-automorphism $\sigma^*\in\Aut(\C^*)$ extending $\sigma_0$ fixing $J$ pointwise.   
 In particular,  $\sigma^*$ permutes the set $\C$ while fixing $MJ$ pointwise.
Hence,  $\sigma^*$ is an $L(U)$-automorphism of $\C^*_U$ fixing $MJ$ pointwise and
$\sigma^*(\ebar_{i_j})=\ebar_{k_j}$ for each $j$.   
Thus,   $I$ is $MJ$-indiscernible in the expanded language.
\qed
\end{proof}

Putting all of this together, in the expanded $\C^*_U$, $I$ and $J$ are mutually indiscernible over $M$, with neither $I$ nor $J$
$L(U)$-indiscernible over $Mc$.  Thus, $\hbox{dp-rk}_{L(U)}\tp_{L(U)}(c/M)\ge 2$, contradicting dp-minimality of $Th(\C^*_U)$.
\qed
\end{proof}  

We understand that Chernikov labels the condition below {\em BS-triviality} (Baldwin-Shelah).

\begin{Theorem}  \label{arbbase}
Suppose $\fg {\abar} B C$ and $c$ is any singleton.  Then either $\fg {\abar c} B C$ or $\fg {\abar} B {C c}$.
\end{Theorem}

\begin{proof}   To begin, we may assume $\tp(\abar/\C)$ does not fork over $B$.  [Why?  By the non-forking, choose a global
$p(\xbar)\in S(\C)$ extending $\tp(\abar/BC)$ that does not fork over $B$.  Choose any $\abar'$ realizing $p(\xbar)$ and choose any $c'$ such that $\tp(\abar c/BC)=\tp(\abar' c'/BC)$.   Apply the Theorem to $\abar'$ and then take an automorphism of $\C^*$ fixing $BC$ sending $\abar'c'$ to $\abar c$.]

Let $p(\xbar)=\tp(\abar/\C)$.   By Fact~\ref{charforkinv} $p(\xbar)$ is $M$-invariant for every small $M\supseteq B$.

Let $q(y)=\tp(c/\C)$ (this conceivably could be algebraic).  
There are now two cases:

\medskip
\noindent{\bf Case 1.}  $q(y)$ does not fork over $B$.
\medskip

In this case, choose any small $M\supseteq B$.  As both $p(\xbar)$ and $q(y)$ are $M$-invariant, $\tp(\abar c/\C)$ is $M$-invariant as well by Corollary~\ref{iterateMinv}.  As $M$ is arbitrary, we have $\fg {\abar c} B {\C}$ hence $\fg {\abar c} B C$.

\medskip
\noindent{\bf Case 2.}  $q(y)$ forks over $B$.
\medskip

By Fact~\ref{charforkinv} choose a small $M\supseteq B$ with $q(y)$ not $M$-invariant.   Thus, for any $a_i\in\abar$,
by Proposition~\ref{MI} we have $\tp(a_i/\C c)=((p\mr{x_i})^*)|\C c$.  In particular, $\tp(a_i/\C c)$ does not fork over $M$ for each $i$.  
So, by  Corollary~\ref{triviality},  $s:=\tp(\abar/\C c)$ does not fork over $M$.  
Since $s$ does not fork over $M$, choose $s^*\in S(\C^*)$ extending $s$, also non-forking over $M$.     
Also, as $p:=\tp(\abar/\C)$ does not fork over $B$, choose an  extension $p^*\in S(\C^*)$
of $p$ and non-forking over $B$.  As $B\subseteq M$, the type $p^*$ is also $M$-invariant.  Both $p^*$ and $s^*$ restrict to $p$ over $\C$, and since  $\C$ is full for $M$, Fact~\ref{fullstationary} gives $p^*=s^*$.  Consequently, $s=p^*|\C c$, which does not fork over $B$.  Restricting to $BC c$ gives $\fg {\abar} B {C c}$.
\end{proof}


\begin{Remark} {\em 
It is curious that for a fixed global type $q=\tp(c/\C)$, the proof works uniformly for every $p=\tp(\bar a/\C)$ not forking over $B$.  If $q$ does not fork over $B$, then $\fg {\abar c} B {\C}$; otherwise, it gives $\fg {\abar} B {\C c}$.  The two conclusions need not be mutually exclusive.
}
\end{Remark}

With this in hand we easily get the analogue of Corollary~\ref{triviality} for arbitrary base sets.

\begin{Corollary}  \label{trivoverB}     [$T$ monadically NIP]   If $\abar=(a_1,\dots,a_n)$ and $\tp(a_i/BC)$ does not fork over $B$ for every $i$, then
$\tp(\abar/BC)$ does not fork over $B$.
\end{Corollary}

\begin{proof}  By induction on $n$.  For $n=1$ there is nothing to prove, so assume $\fg {\abar} B C$ and $\fg {a^*} B C$ for some singleton $a^*$.    By Theorem~\ref{arbbase}, either $\fg {\abar a^*} B C$ holds, in which case we are done.
Or, $\fg {\abar} B {C a^*}$ and we are done again by left-transitivity of non-forking. 
\qed
\end{proof}

As well, we get the following `transitivity of forking for 1-types,' which was known by Baldwin and Shelah \cite{BS} in the monadically stable setting.

\begin{Corollary}  \label{forkingequivrelation}  Suppose $B$ is any small set and $a,b,c$ are singletons.   Then
$$\nfg a B b \quad \hbox{and} \quad \nfg b B c \quad \Longrightarrow \nfg a B c$$
\end{Corollary}

\begin{proof}  By way of contradiction, assume $\fg a B c$.  Then by Theorem~\ref{arbbase}, either $\fg {ab} B c$ holds or
$\fg a B {cb}$ holds, both of which contradict our assumptions.
\qed
\end{proof}

\section{An application}  \label{RRR}
In \cite{SS} Simon and Starchenko introduced the notion of a dp-minimal theory having density of definable types.   They gave a condition (D) and proved that if 
$T$ is dp-minimal and satisfies condition (D), then $T$ has density of definable types.  Here, we use the material above to prove that every monadically NIP theory
also has density of definable types.  The proof here is rather different than in \cite{SS}, and we include an example of a monadically NIP theory that does not have condition (D).

\begin{Definition}  {\em  A theory $T$ has {\em density of definable types among non-forking types} if, for every small model $M\preceq\C$ and every formula
$\phi(\xbar,\cbar)$, $\phi(\xbar,\cbar)$ does not fork over $M$ if and only if there is an $M$-definable global type $p(\xbar)$ with $\phi(\xbar,\cbar)\in p$.\\
A theory $T$ has property (D) if, for all small sets $A$ and all consistent $L(A)$-formulas $\phi(x)$ in one variable, there is an $A$-definable complete type $p\in S_x(A)$ extending $\phi(x)$.
}
\end{Definition}

Theorem~5 of \cite{SS} states that if $T$ is dp-minimal with property (D), then $T$ has density of definable types among non-forking types.  

The one-line proof of the following Theorem is to replace the finitely satisfied part of a block decomposition by a tuple from $M$ and retain the definable block.  

\begin{Theorem}  If $T$ is monadically NIP, then $T$ has density of definable types among non-forking types.  
\end{Theorem}

\begin{proof}  As any $M$-definable global type $p$ does not fork over $M$, right to left is obvious.  
%
%
%
%
For left to right, since $\phi(\xbar,\cbar)$ does not fork over $M$, choose an $M$-invariant, global $p(\xbar)\in S_{\xbar}(\C)$ containing $\phi(\xbar,\cbar)$ and choose a realization $\abar$ of $p\mr{M\cbar}$.   Apply  Theorem~\ref{partition} to get a partition $\abar=(\fbar,\dbar)$ as there.   
%

Since $\tp(\fbar/M\cbar\dbar)$ is  finitely satisfied in $M$,  there is an $\mbar\in M^{\lg(\fbar)}$ so
that $\phi(\mbar,\dbar,\cbar)$ holds.    As $\tp(\dbar/M\cbar)$ is $M$-definable, let $q(\xbar_D)\in S(\C)$ be its unique
global heir.   The global type $r(\xbar)$ generated by $\{\xbar_F=\mbar\}\cup q(\xbar_D)$ works.
\qed
\end{proof}

\begin{Example}  A monadically NIP theory $T$ that fails property (D).
\end{Example}

Let $T=Th(M)$, where $M=(2^{\le\omega},\triangleleft)$ is the full binary tree of height $\omega+1$, where $\triangleleft$ is a binary relation interpreted as `proper initial segment'.  As any tree is a substructure of sufficiently large dense tree and since
`being monadic NIP' is preserved under substructures (see e.g., Lemma~2.4 of \cite{BLE}), $T$ is monadically NIP.  

Let $A=2^{<\omega}$ be the finite nodes of $M$.  Note that $M\setminus A$ is defined by the 1-formula $\phi(x)$ asserting that there are no elements properly extending
$x$.   Let $p(x)\in S_x(A)$ be any complete type extending $\phi(x)$.  We argue that $p$ is not $A$-definable.  In particular, there is no $L(A)$-formula
$\theta(y,\abar)$ that could be $d\psi(y)$ for the  formula $\psi(x,y):= y\triangleleft x$.   To see this, we first describe  $$B_p:=\{b\in A:\psi(x,b)\in p\}$$
As $M$ is a tree, for every $n\in\omega$, the cones above the nodes $2^n$ partition $M\setminus A$,  so $B_p\cap 2^n$ has exactly one element.    Moreover, the elements of $B_p$ are linearly ordered by $\triangleleft$.  Thus, $B_p$ describes an  $\omega$-branch through $A$.  

By way of contradiction, suppose  there were an $L(A)$-formula $\theta(y,\abar)$ so that
for every $b\in A$, $M\models \theta(b,\abar)$ if and only if $\psi(x,b)\in p$.   Since the parameters are from $A$, choose an integer $k$ so that the finite sequence 
$\abar\subseteq 2^{< k}$.
From the comments above, choose $b^*\in B_p\cap 2^k$.  Now let $b,b'$ be the two immediate successors of $b^*$ in $A$.  From above, exactly one of these is in $B_p$, say $b\in B_p$, but $b'\not\in B_p$.  Since $M$ is highly homogeneous  there is an automorphism $\sigma\in Aut(M)$ with $\sigma(b)=b'$, but $\sigma\mr{2^{\le k}}=id$.   In particular, $\sigma$ fixes $\abar$ pointwise, so $M\models\theta(b,\abar)\leftrightarrow\theta(\sigma(b),\abar)$.   As $b\in B_p$ while $b'\not\in B_p$, $\theta(y,\abar)$ is not a definition of
$\psi(x,y)$ in $p$.


\begin{thebibliography}{99}

\bibitem{BS}   J.T.\ Baldwin and S.\ Shelah,  Second-order quantifiers and the complexity of theories,
{\em  Notre Dame J.\ Formal Logic}  {\bf  26} (1985), no.\ 3, 229--303.

\bibitem{BL1}  S.\ Braunfeld and M.C.\ Laskowski, Characterizations of monadic NIP, {\em Trans.\ Amer.\ Math.\ Soc.\  Series B}  {\bf  8} (2021), 948--970.

\bibitem{BL2}  S.\ Braunfeld and M.C.\ Laskowski, Corrigenda to ``Characterizations of monadic NIP", {\em Trans.\ Amer.\ Math.\ Soc.\  Series B}  
 {\bf 11} (2024), 1226--1232.
 
\bibitem{BLE}  S.\ Braunfeld and M.C.\  Laskowski, Existential characterizations of monadic NIP, {\em Communications of the American Mathematical Society} {\bf 6} (2026), 244--267. 


\bibitem{LP}  D.\ Lascar and B. Poizat, An introduction to forking,
{\em Journal of Symbolic Logic}  {\bf  44} (1979), no.~3, 330--350.


\bibitem{PierreBook}  P.\ Simon,  {\em A guide to NIP theories,} Lecture Notes in Logic, 44. Association for Symbolic Logic, Chicago, IL; Cambridge University Press, Cambridge, 2015.

\bibitem{dpInv}  P.\ Simon, Dp-minimality: invariant types and dp-rank, {\em Journal of Symbolic Logic}  {\bf  79} (2014), no.~4, 1025--1045.  


\bibitem{SS}  P.\ Simon and S.\ Starchenko, On forking and definability of types in some dp-minimal theories, {\em  Journal of Symbolic  Logic}  {\bf 79} (2014), no.~4, 1020--1024. 






\end{thebibliography}
 \end{document}